\documentclass[12pt]{article}
\usepackage{amsmath}
\usepackage{amsfonts}
\usepackage{amssymb}
\usepackage{graphicx}
\usepackage{setspace}
\newtheorem{theorem}{Theorem}[section]
\newtheorem{lemma}{Lemma}[section]
\newtheorem{corollary}{Corollary}[section]

\newtheorem{definition}{Definition}[section]
\newenvironment{proof}{\smallskip\noindent{\it Proof}}{$\Box$}
\numberwithin{equation}{section}
\newcommand{\bdelta}{{\boldsymbol \delta}}

\def\by{{\boldsymbol y}}

\def\bS{{\boldsymbol S}}

\def\bY{{\boldsymbol Y}}

\def\bz{{\boldsymbol z}}
\def\be{{\boldsymbol e}}
\def\bv{{\boldsymbol v}}

\def\bx{{\boldsymbol x}}
\def\bX{{\boldsymbol X}}

\def\cM{{\mathcal M}}
\def\cL{{\mathcal L}}
\def\cF{{\mathcal F}}
\def\cJ{{\mathcal J}}
\def\cG{{\mathcal G}}
\def\bbP{{\mathbb P}}

\graphicspath{%
    {converted_graphics/}
    {/}
}
\begin{document}

\title{Prediction in Riemannian metrics \\derived from divergence functions} 
\author{Henryk Gzyl,\\
\noindent 
Centro de Finanzas IESA, Caracas, Venezuela.\\
 henryk.gzyl@iesa.edu.ve}

\date{}
 \maketitle

\setlength{\textwidth}{4in}
\vskip 1 truecm
\baselineskip=1.5 \baselineskip \setlength{\textwidth}{6in}
\begin{abstract}
Divergence functions are interesting discrepancy measures. Even though they are not true distances, we can use them to measure how separated two points are. Curiously enough, when they are applied to random variables, they lead to a notion of best predictor that coincides with usual best predictor in Euclidean distance. From a divergence function, we can derive a Riemannian metric, which leads to a true distance between random variables, and in which best predictors do not coincide with their Euclidean counterparts. It is the purpose of this note to point out that there are many interesting ways of measuring distance between random variables, and to study the notion of best predictors that they lead to.
\end{abstract}


\begin{spacing}{0.01}
   \tableofcontents
\end{spacing}

\section{Introduction and Preliminaries}
In \cite{Br}, Bregman introduced an iterative procedure to find points in an intersection of convex sets. At each step, the next point in the sequence is obtained by minimizing an objective function, that can be described as the vertical distance of the graph of the function to the tangent  plane through the previous point.
If $\cM$ is a convex set in some $\mathbb{R}^K,$ and $\Phi:\cM\to\mathbb{R}$ is a strictly convex, continuously differentiable function, the {\it divergence function} that it defines is specified by
\begin{equation}\label{breg}
\bdelta_\Phi(\bx,\by)^2 = \Phi(\bx)-\phi(\by)-\langle(\bx-\by),\nabla\Phi(\by)\rangle.
\end{equation}
In Bregman's work, $\Phi(\bx)$ was taken to be the Euclidean square norm $\|\bx\|^2.$ The concept was eventually extended, even to the infinite dimensional case, and now plays an important role in many applications. For example, in clustering. classification analysis  and machine learning as in Banerjee et al. \cite{BGW}, Boisonnat el al. \cite{BNN}, Banerjee et al. \cite{BDGMM}. Fisher \cite{F}. It plays a role in optimization theory as in Baushke and Borwein \cite{BB}, Baushke and Lewis \cite{BL}, Baushke and Combettes \cite{BC},  Censor and Reich \cite{CR}, Baushke et al. \cite{BBC} and Censor and Zaknoon \cite{CZ}, or to solve operator equations as in Butnariu and Resmerita \cite{BR}, in approximation theory in Banach spaces as in Baushke and Combettes \cite{BC} or Li et al. \cite{LSY}. In applications of geometry to statistics and information theory as in Amari and Nagaoka \cite{AN},  Csisz\'ar \cite{Cs}, Amari and Cichoski \cite{AC}, Calin and Urdiste \cite{CU} or Nielsen \cite{N}. These are just a small sample of the many references to applications of Bregman functions, and the list cascades rapidly.

Is is a well known, and easy to verify fact, that 
$$\bdelta_\Phi(\bx,\by)^2 \geq 0,\;\;\mbox{and}\;\; \bdelta_\phi(\bx,\by)^2=0 \Leftrightarrow \bx=\by.$$
Thus our choice of notation is consistent. But as $\bdelta$ is not symmetric, nor does it satisfy the triangular inequality, it can not be a distance on $\cM.$
Let now $(\Omega,\cF,\bbP)$ be a probability space such that $\cF$ is complete (contains all sets of zero $\bbP$ measure). By $\cL_p, p=1,2$ we shall denote the usual classes of $\bbP$ integrable or square integrable functions, identified up to sets of measure zero. The notion of divergence can be extended to random variables as follows
\begin{definition}\label{divrv}
Let $\bX,\bY$ be $\cM$-valued random variables such that $\Phi(\bX),$ $\Phi(\bY)$ and $\nabla\Phi(\bY)$ are in $\cL_2.$ The divergence between $\bX$ and $\bY$ is defined by
$$\Delta_\Phi(\bX,\bY)^2 = E[\bdelta_\Phi(\bX,\bY)^2] = \int_\Omega\bdelta_\Phi(\bX(\omega),\bY(\omega))^2d\bbP(\omega).$$
\end{definition}
Clearly, $\Delta_\Phi(\bX,\bY)$ is neither symmetric nor satisfies the triangle inequality. But as above, we also have
$$\Delta_\Phi(\bX,\bY)^2 \geq 0\;\;\mbox{and}\;\;\Delta_\Phi(\bX,\bY)^2=0 \Leftrightarrow \bX=\bY\;\;{a.s}\;\bbP$$
we can think of it as a pseudo distance, cost or penalty function on $\cL_p.$

The motivation for this work comes from two directions. On the one hand, there is the fact that for Bregman divergences there is a notion of best predictor, and this best predictor happens to be the usual conditional expectation. To put it in symbols
\begin{theorem}\label{bpdiv}
Let $\bX\in\cL_2$ and let $\cG\subset\cF.$ Then the solution to the problem
$$\inf\{\Delta_\Phi(\bX,\bY)^2\,|\,\bY\in\cL_2(\cG)\}$$
is given by $E[\bX | \cG].$
\end{theorem}
For the proof the reader can consult Banerjee et al., \cite{BGW} or  Fisher's \cite{F}. The other thread comes from Gzyl's \cite{GH}, where a geometry on the convex cone of strictly positive is considered. That geometry happens to be derivable from a divergence function, and it leads to a host of curious variations on the theme of best predictor, estimation, laws of large numbers and central limit theorems. The geometry considered there is that induced by the logarithmic distance, which makes $(0,\infty)^d$ a Tits-Bruhat space, which happens to be a special commutative version of the theory explained in Lang \cite{L}, Lawson and Lin \cite{LL}, Mohaker \cite{Moh} and Schwartzman \cite{Sch}.

We should mention that the use of differential geometric methods in \cite{AN}, or \cite{CU} and the many references cited therein, is different from the one described below. They consider geometric structure either on the class of probabilities on a finite set, or in the space of parameters characterizing a (usually exponential) family of distributions.  Here we analyze how the geometry on the set in which the random variables take value, determines the nature of the standard estimation and prediction process.

From now on we shall suppose that $\cM = \cJ^K,$ where $\cJ$ is a bounded or unbounded interval in $\mathbb{R}.$ We shall denote by $\phi:\cJ\to\mathbb{R}$ a strictly convex, three times continuously differentiable function, and define 
\begin{equation}\label{convf}
\Phi(\bX) = \sum_{i=1}^K \phi(x_i)
\end{equation}

\subsection{Some standard examples}
In the next table we list five standard examples. The list could be quite longer, but the examples chosen because the in some of the cases the distance between random variables associated to the divergence bounds their divergence from above, whereas in the other, it is bounded by the divergence from above. The examples are displayed in Table \ref{tab1}.
\begin{table}[h!]
\centering
\begin{tabular}{|c|c|}\hline
Domain & $\phi$ \\\hline 
$\mathbb{R}$ & $x^2$\\\hline
$\mathbb{R}$ & $e^x$ \\\hline
$\mathbb{R}$ & $e^{-x}$ \\\hline
$(0,\infty)$ & $x\ln x$ \\\hline
$(0,\infty)$ & $-\ln x$\\ \hline
\end{tabular}
\caption{Standard convex functions used to generate Bregman divergences}
\label{tab1}
\end{table}

\subsection{Organization of the paper}
We have established enough notations to describe the contents of the paper. In Section 2 we start from the divergence function on $\cM$ and derive a metric tensor $g_{i,j}$ from it. We then solve the geodesic equations to compute the geodesic distance $d_\phi(\bx,\by)$ between any two points $\bx,\by\in\cM,$ and we compare it with the divergence $\bdelta_\phi(\bx,\by)$ between the two points. We shall see that there are cases in which one of them dominates the other for any pair of points. 

The Riemannian distance between points in $\cM$ induces a distance between random variables taking values in there. In Section 3 we come to the main theme of this work, that is, to the computations of best predictors when the distance between random variables is measured in the induced Riemannian distance. We shall call such best predictors the $d-$mean and the $d$-conditional expectation and denote them by $E_d[X]$ and, respectively, $E_d[\bX|\cG].$ In order to compare these to the best predictor in divergence, we use the prediction error as a comparison criterion. It is at this point at which the comparison results established in Section 2 come in.

In Section 4 we take up the issue of sample estimation and its properties. We shall see that the standard results hold for the $d$-conditional expectation as well. That is, we shall see that the estimator of the $d$-mean and that of the $d$-variance, are unbiased and converge to their true values as the size of the sample becomes infinitely large. In Section 5 we shall consider the arithmetic properties of the $d$-conditional expectation when there is a commutative group structure on $\cM.$ In  Section 6 we collect a few final comments, and in Appendix 7, we present one more derivation of the geodesic equations.

\section{Riemannian metric induced by $\phi$}
The direct connection between $\Phi$-divergences stems from the fact that a strictly convex, at least twice differentiable function has a positive definite Hessian matrix. Even more,  metric derived from a ``separable'' $\Phi$  is diagonal, that is
\begin{equation}\label{metric}
g_{i,j} = \frac{\partial^2 \Phi(\bx)}{\partial x_i\partial x_j} = \phi''(x_i)\delta_{i,j}
\end{equation}
\noindent Here we use $\delta_{i,j}$ for the standard Kronecker delta and we shall  not distinguish between covariant and contravariant coordinates. This may make the description of standard symbols in differential geometry a bit more awkward.

All these examples have an interesting feature in common. The convex function defining the Bregman divergence is three times continuously differentiable, and defines a Riemannian metric in its domain by $g_{i,j}(\bx)=\phi''(x_i)\delta_{i,j}.$ The equations for the geodesics in this metric are separated. It is actually easy to see that for each $1\leq i \leq d,$ the equation for defining the geodesic which at time $t=0$ starts from $x_i$ and end at $y_i$ at time $t=1,$ is the solution to
\begin{equation}\label{geoeq}
\phi''(x_i(t))\ddot{x}_i(t) + \frac{1}{2}\phi'''(x_i(t))\dot{x}_i^2(t) = 0\;\;x_i(0)=x_i,\;x_i(1)=y_i.
\end{equation}
Despite the fact that it is easy to integrate this equation rapidly, we show how to integrate this equation in a short appendix at the end. Now denote by $h(x)$ as a primitive of $(\phi''(x))^{1/2},$ that is
\begin{equation}\label{chanvar}
h(x) =\int^x\big(\phi''(t)\big)^{1/2}dt.
\end{equation}
therefore, it is strictly positive by assumption, it is invertible because it is strictly increasing. If we put $H=h^{-1},$ for the compositional inverse of $h,$  we can write the solution to (\ref{geoeq}) as
\begin{equation}\label{solgeo}
x_i(t) = H\Big(h(x_i) + k_it\Big)\;\;\;0\leq t \leq 1,\;\;i=1,...,K. 
\end{equation}
The $k_i$ are integration constants, which using the condition that $x_i(0)=x_i, x_i(1)=y_i$ turn out to be $k_i=h(y_i)- h(x_i).$ Notice now that the distance between $\bx$ and $\by$ is given by
\begin{equation}\label{geodis}
d_\phi(\bx,\by) = \int_0^1\Big(\sum_{i=1}^K\phi''(x_i(t))\dot{x}_i^2(t)\Big)^{1/2}dt.
\end{equation}
It takes a simple computation to verify that
\begin{equation}\label{geodes2}
d_\phi(\bx,\by) = \Big(\sum_{i=1}^K k_i^2\Big)^{1/2} = \Big(\sum_{i=1}^K h(y_i) - h(x_i))^2\Big)^{1/2}
\end{equation}
For not to introduce more notation, we shall use the symbol $h$ to denote the map $h:\cM \to \mathbb{R}^K,$ defined by $h(\bx)_i=h(x_i).$ Notice that $h$ is isometry between $\cM$ and its image in $\mathbb{R}^K,$ when the distance in the former is $d_phi$ and in the Later is the Euclidean distance. Therefore geometric properties in $\mathbb{R}^K$ have a counterpart in $cM.$

Observe as well that the special form of (\ref{solgeo}) and (\ref{geodes2}) allows us to represent the middle point between $\bx$ and $\by$ easily. As a mater of fact,we have
\begin{lemma}\label{midpt}
With the notations introduced above,  observe that if we put $z_i=\zeta(\bx,\by)_i = H\Big(\frac{1}{2}\big(h(x_i)+h(y_i)\big)\Big),$ then
$$d_\phi(\bx,\bz) = d_\phi(\by,\bz) = \frac{1}{2}d_\phi(\bx,\by) = \frac{1}{2}\Big(\sum_{i=1}^K(h(y_i) + h(x_i))^2\Big)^{1/2}.$$
\end{lemma}

\subsection{Comparison of Bregman and Geodesic distances}
Here we shall examine the relationship between the $\phi$-divergence and the distance induced by $\phi.$ Observe to begin with, for any three times continuously differentiable function, we have $\phi(y)-\phi(x)=\int_x^y\phi'(u)du.$ Applying this once more under the integral sign, and rearranging a bit, we obtain
\begin{equation}\label{basic}
\phi(y) - \phi(x) -(y - x)\phi'(x) = \int_x^y\phi''(u)(y - u)du.
\end{equation}
Notice that the left hand side is the building block of the $\phi$-divergence. To make the distance (\ref{geodes2}) appear on the right hand side of (\ref{basic}), we rewrite it as follows. Use the fact that $h'(x)=(\phi''(x))^{1/2},$ and invoke the previous identity applied to $h$ to obtain
$$\int_x^y h'(u)(y-u)dh(u) = \int_x^y \Big(h(y) - h(u) -\int_u^y h''(\xi)(y-\xi)d\xi\Big)dh(u).$$
Notice now that
$$\int_x^y \Big(h(y) - h(u)\Big)dh(u) = \frac{1}{2}\big(h(y) - h(x)\big)^2.$$
With this, it is clear that
$$\phi(y)-\phi(x)-(y-x)\phi'(x) = \frac{1}{2}\big(h(y) - h(x)\big)^2 - \int_x^y\Big(\int_u^y h''(\xi)(y-\xi)d\xi\Big)dh(u)$$
We can use the previous comments to complete the proof of the following result.
\begin{theorem}\label{compare}
With the notations introduced above, suppose furthermore that $\phi'''(x)$ (and therefore $h''(x)$) has a constant sign. Then
\begin{eqnarray}\label{compare2}
\bdelta_\phi(\by,\bx)^2 \leq \frac{1}{2}d_\phi(\by,\bx)^2,\;\;\;\mbox{if}\;\;\;\phi''' > 0.\\\label{compare2.1}
\bdelta_\phi(\by,\bx)^2 \geq \frac{1}{2}d_\phi(\by,\bx)^2,\;\;\;\mbox{if}\;\;\;\phi''' < 0.
\end{eqnarray}
\end{theorem}

This means that, for example, in the first case, a minimizer with respect to the geodesic distance, yields a smaller approximation error that the corresponding minimizer with respect to the divergence. The inequalities in Theorem \ref{compare} lead to the following result
\begin{theorem}\label{compmeans1}
Let $\{\bx_1,...,\bx_n\}$ is be set of points in $\cM,$ and $\bx_\phi^*$ and $\bx_d^*$ respectively denote the points in $\cM$ closer to that set in $\phi$-divergence and geodesic distance. Then, for example, when (\ref{compare2.1}) holds, 
$$\sum_{i=1}^n \bdelta_\phi(\bx_i,\bx_\phi^*)^2 \geq \frac{1}{2}\sum_{i=1}^n d_\phi(\bx_i,\bx_d^*)^2,$$
\end{theorem}
\begin{proof}$\,$
If (\ref{compare2.1}) holds, then $\sum_{i=1}^n \bdelta_\phi(\bx_i,\bx)^2 \geq \frac{1}{2}\sum_{i=1}^n d_\phi(\bx_i,\bx)^2$ for any $\bx\in\cM.$ Therefore, to begin with, since $\bx_d^*$ minimizes the right hand side, we have  $\sum_{i=1}^n \bdelta_\phi(\bx_i,\bx)^2 \geq \frac{1}{2}\sum_{i=1}^n d_\phi(\bx_i,\bx_d^*)^2$ for any $\bx\in\cM.$ Now minimizing with respect to $\bx$ on the left hand side of this inequality we obtain the desired result.
\end{proof}

\noindent That is, the approximation error is smaller for the minimizer computed with the geodesic distance than that computed for the divergence. We postpone the explicit computation of $\bx_d^*$ to Section 4, when we show how to compute sample estimators.

{\bf Comment} Note that we can think of (\ref{basic}) as a way to construct a convex function starting from its second derivative. What the previous result asserts that if we start from a positive but strictly decreasing function, we generate a divergence satisfying (\ref{compare2.1}), whereas if we start from a positive and strictly increasing function, we generate a divergence satisfying (\ref{compare2}). This is why we included the second and third examples. Even though they would seem to be related by a simple reflection at the origin, their predictive properties are different.

Note that when $\phi'''$ is identically zero as in the first example of the list in Table \ref{tab1}, the two distances coincide. This example is the first case treated in the examples described below. The other examples are standard examples used to define Bregman divergences.

Note as well that when $\phi(x)=x^p$ with $1<p<2$ the derived distance has smaller prediction error than that of the prediction error in divergence, whereas when $2<p$ the prediction error in divergence is smaller than the prediction error in its derived distance. And we already noted that for $p=2$ both coincide. But to compare the $d$-metric with the Euclidean metric does not seem an easy task.

\subsection{Examples of distances related to a Bregman divergence}
\subsubsection{Case 1: $\phi = x^2/2$}
In this case $\phi''(x)=1$ and $\phi'''(x)=0.$ The geodesics are the straight lines in $\mathbb{R}^K$ and the induced distance is the standard Euclidean distance
$$d_\phi(\bx,\by)^2 = \sum_{i=1}^K (x_i - y_i)^2.$$
\subsubsection{Case 2: $\phi(x)=e^x$}
Now $\phi''(x)=\phi'''(x)=e^x.$ The solution to the geodesic equation (\ref{geoeq}) is given by 
$x_i(t)=2\ln\Big(e^{x_i/2} + k_it\Big),\;i=1,...,K$ and therefore $k_i=e^{y_i/2}-e^{x_i/2}.$ The geodesic distance between $\bx$ and $\by$ is given by
$$d_\phi(\bx,\by)^2 = \sum_{i=1}^K (e^{y_i/2}-e^{x_i/2})^2.$$
\subsubsection{Case 3: $\phi(x)=e^{-x}$} 
Now $\phi''(x)=e^{-x}$ but $\phi'''(x)=-e^{-x}.$ The solution to the geodesic equation (\ref{geoeq}) is given by 
$x_i(t)=-2\ln\Big(e^{-x_i/2} + k_it\Big),\;i=1,...,K$ and therefore $k_i=e^{-y_i/2}-e^{-x_i/2}.$ The geodesic distance between $\bx$ and $\by$ is given by
$$d_\phi(\bx,\by)^2 = \sum_{i=1}^K (e^{-y_i/2}-e^{-x_i/2})^2.$$
\subsubsection{Case 4: $\phi(x)=x\ln x$}
This time our domain is $\cM=(0,\infty)^K$ and $\phi''(x)=1/x$ whereas $\phi'''(x)=-1/x^2.$  The solution to the geodesic (\ref{geoeq}) is given by $x_i(t) =\Big(\sqrt{x_i} + k_it\Big)^2,\;i=1,...,K$ where $k_i=\sqrt{y_i}-\sqrt{x_i}.$ Therefore, the geodesic distance between $\bx$ and $\by$ is
$$d_\phi(\bx,\by)^2 = \sum_{i=1}^K (\sqrt{y_i} - \sqrt{x_i})^2.$$
This look similar to the Hellinger distance used in probability theory. See Pollard's \cite{P} 
\subsubsection{Case 5: $\phi(x)=-\ln x$}
To finish, we shall consider another example on $\cM=(0,\infty)^K.$ Now, $\phi''(x)=1/x^2$ and $\phi'''(x)=-1/x^3.$ The geodesics turn out to be given by $x_i(t)=x_ie^{kt}$ where $k_i=\ln\big(y_i/x_i)$ which yields the representation $\bx(t)=\bx^{(1-t)}\by^{t}.$ Recall that all operations are to be understood componentwise ($d$-vectors are  function on $[1,...,K]$). The distance between $\bx$ and $\by$ is now given by
$$d_\phi(\bx,\by)^2 = \sum_{i=1}^K (\ln y_i - \ln x_i)^2.$$

\subsection{The semi-parallelogram law of the geodesic distances}
As a consequence of Lemma \ref{midpt} and the way the geodesic distances are related to the Euclidean distance through a bijection, we have the following result:
\begin{theorem}\label{spl}
With the notations introduced in the four examples listed above, the sets $\cM$ with the corresponding geodesic distances, satisfy the semi-parallelogram law. that is in all four cases considered, for any $\bx,\by\in\cM,$ there exists a $\bz$ obtained as in Lemma \ref{midpt}, such that for any $\bv\in\cM$ we have
$$d_\phi(\bx,\by)^2 + 4d_\phi(\bv,\bz)^2 \leq d_\phi(\bv,\bx)^2 + 2d_\phi(\bv,\by)^2.$$
\end{theorem}
That is, for separable Bregman divergences, the induced Riemannian geometry is a Tits-Bruhat geometry. The semi-parallelogram property is handy in proofs of uniqueness.

\section{$\cL_2$-conditional expectations related to Riemannian metrics derived from a Bregman divergences}
As we do not have a distinguished point in $\cM$ which is the identity with respect to a commutative operation on $\cM,$ in order to define a squared norm for $\cM$-valued random variables we begin by introducing the following notation.
\begin{definition}\label{norm}
We shall say that a $\cM$-valued random variable is integrable or square integrable, and write $\bX\in\cL^\phi_p,$ (for $p=1,2$) whenever
$$D_\phi(\bX,\bx_0)^p = E\Big(d_\phi(\bX,\bx_0)\Big)^p] < \infty$$
\noindent for some $\bx_0\in\cM.$ It is clear from the triangular inequality that this definition is independent of $\bx_0.$ 
\end{definition}
But more important in the following simple result
\begin{lemma}\label{simple2}
With the notations introduced above, from (\ref{geodes2}) if follows that $\bX\in\cL^\phi_p,$ is equivalent to $h(\bX)\in\cL_p.$ 
\end{lemma} 

With identity (\ref{geodes2}) in mind, it is clear that  $\bX,\bY\in\cL^\phi_2,$ the distance on $\cM$ extends to a distance between random variables by
\begin{equation}\label{dist3}
D_\phi(\bX,\bY)^2 = E[\Big(d_\phi(\bX,\bY)\Big)^2] = E[\|h(\bX) - h(\bY)\|^2].
\end{equation}
Now that we have this definition in place, the extension of Theorem (\ref{compare}) to this case can be stated as follows.
\begin{theorem}\label{compare3}
For any pair $\bX,\bY$ of $\cM$-valued random variables such that the quantities written below are finite, we have
\begin{eqnarray}\label{compare3.1}
\Delta_\phi(\bY,\bX)^2 \leq \frac{1}{2}D_\phi(\bY,\bX)^2,\;\;\;\mbox{if}\;\;\;\phi''' > 0.\\\label{compare3.2}
\Delta_\phi(\bY,\bX)^2 \geq \frac{1}{2}d_\phi(\bY,\bX)^2,\;\;\;\mbox{if}\;\;\;\phi''' < 0.\\\nonumber
\end{eqnarray}
\end{theorem}

We can now move on to the determination of best predictors in the $D_\phi$ distance.
\begin{theorem}\label{dcondexp}
Let $(\Omega,\cF,\bbP)$ be a probability space and let $\cG$ be a sub-$\sigma$-algebra of $\cF.$ Let $X$ be a $\cM$-valued random variable such that $h(\bX)$ is $\bbP$-square integrable. Then 
$$E_d[\bX|\cG] = H\Big(E[h(\bX)|\cG]\Big).$$
\end{theorem} 
Keep in mind that the both $h$ and its inverse $H$ act componentwise. This theorem has a curious corollary, to wit:
\begin{corollary}[Intertwining]
With the notations in the statement of the last theorem, we have
$$h\big(E_d[\bX|\cG]\Big) = \Big(E[h(\bX)|\cG]\Big).$$
\end{corollary}
\subsection{Comparison of prediction errors}
As a corollary of Theorem \ref{compare3} to compare the prediction errors in the $d$-metric or in divergence. 
\begin{theorem}\label{prederror}
With the notations of Theorem \ref{compare3}, we have
\begin{eqnarray}\label{compare3.1}
\Delta_\phi(\bX,E[\bX|\cG])^2 \leq D_\phi(\bX,E_d[\bX|\cG])^2,\;\;\;\mbox{if}\;\;\;\phi''' > 0.\\\label{compare3.2}
\Delta_\phi(\bX,E[\bX|\cG])^2 \geq D_\phi(\bX,E_d[\bX|\cG])^2,\;\;\;\mbox{if}\;\;\;\phi''' < 0.\nonumber
\end{eqnarray}
\end{theorem}
The proof is simple. For the first case say, begin with (\ref{compare3.2}) and since the right hand side decreases by replacing $\bY$ with $E_d[\bX|\cG]$ we have $\Delta_\phi(\bX,\bY])^2 \geq D_\phi(\bX,E_d[\bX|\cG])^2$ for any $\bY$ with the appropriate integrability. Now, minimize the left hand side of las inequality with respect to $\bY$ to obtain the desired conclusion.

\subsection{Examples of conditional expectations}
Even though the contents of the next table are obvious, they are worth recording. There we display the appearance of the conditional expectations of a $\cM$-valued random variable $\bX$ in the metrics derived from the divergences listed in Table \ref{tab1}. 

\begin{table}[h]
\centering
\begin{tabular}{|c|c|c|c|}\hline
Domain & $\phi$ & $h$  &Conditional Expectation\\\hline 
$\mathbb{R}$ & $x^2$ & $x$ & $E[\bX|\cG]$\\\hline
$\mathbb{R}$ & $e^x$ & $e^{x/2}$ & $2\ln\Big(E[e^{\bX/2}|\cG]\Big)$\\\hline
$\mathbb{R}$ & $e^{-x}$ & $e^{-x/2}$ & $2\ln\Big(\frac{1}{E[e^{-\bX/2}|\cG]}\Big)$\\\hline
$(0,\infty)$ & $x\ln x$ & $\sqrt{x}$ & $\Big(E[\sqrt{\bX}|\cG]\Big)^2$ \\\hline
$(0,\infty)$ & $-\ln x$ &  $\ln(x)$ & $\exp\Big(E[\ln(\bX)|\cG]\Big)$\\ \hline 
\end{tabular}
\caption{Expected conditional values in $d_\phi$ metric}
\label{tab2}
\end{table}

The only other information that we have about $\cM$ in this context is that it is a convex set in $\mathbb{R}^d.$ But we do not know if it is closed with respect to any group operation. IN this regard, see Section 5.1. Thus the only properties of the conditional expectations that we can verify at this points are those that depend only on its definition, and on the corresponding property for $h(\bX)$ with respect to $\bbP.$
\begin{theorem}\label{dcondexp1}
With the notations introduced in the previous result, and assuming that all variables mentioned are $D$-integrable we have\\
{\bf 1)} Let $\cF_0=\{\emptyset,\cF\}$ be the trivial $\sigma$-algebra, then $E_d[\bX|\cF_0]=E_d[\bX].$\\
{\bf 2)} Let $\cG_1\subset\cG_2$ be two sub-$\sigma$-algebras of $\cF,$ then $E_d[E_d[\bX|\cG_2]|\cG_1]=E_d[\bX|\cG_1].$
{\bf 3)} If $\bX\in\cG,$ then $E_d[\bX|\cG]=\bX .$\\
As both $h$ and $H$ are defined component wise, and are increasing, we can also verify  the monotonicity properties of the conditional expectations. \\
{\bf 4)} Let $\bX,\bY\in\cM$ with $\bX\leq\bY,$ then $E_d[\bX|\cG] \leq E_d[\bY|\cG].$\\
We do not necessarily have a $\mathbf{0}$ vector in $\cM,$ but a monotone convergence property may be stated as\\
 {\bf 5)} Let $\{\bX_n:n\geq 1\}$ be a sequence in $\cM$ increasing to $\bX\in\cM,$ and suppose that there exist $\cF$-measurable $\bY_1 \leq\bX_n\leq\bY_2$ and $E[h(Y_i)]\in\cL_1.$ Then $E_d[\bX_n|\cG]\uparrow E[\bX|\cG].$   
\end{theorem}
\subsection{A simple application}
Let us consider the following two strictly positive random variables (that is $K=1$ and $\cM=(0,\infty)$):
$$S(2) = S(0)e^{X+Y}\;\;\mbox{and}\;\;S(1)=S(0)e^{X}$$
where $X\sim N(\mu_1,\sigma_2^2)$ and $Y\sim N(\mu_2,\sigma^2_2)$ are two Gaussian, $\rho$-correlated random variables, with $-1<\rho<1.$ It is a textbook exercise to verify that 
$$Y_{|X} \sim N(\mu_2+\rho\frac{\sigma_2}{\sigma_1}(X-\mu_1),\sigma_2^2(1-\rho^2)).$$
If we consider the logarithmic distance on $(0,\infty),$ an application of the results in the previous section, taking into account that $S(1)$ and $X$ generate the same $\sigma$-algebra (call it $\cG$) we have that
$$E_d[S(2)|\cG] = e^{E[\ln(S(2))|\cG]} = S(1)e^{E[Y_{|X}]} = S(1)e^{m}$$
\noindent where we put $m=\mu_2+\rho\frac{\sigma_2}{\sigma_1}(X-\mu_1).$ For comparison note that the predictor in the Euclidean distance is given by
$$E[S(2)|\cG] = S(1)E[e^{Y_{|X}}] = S(1)e^{m+(1-\rho^2)\sigma_2^2/2}$$
According to Theorem \ref{prederror}, the previous one is better than the last  because its variance is smaller. 

A possible interpretation of this example goes as follows. We might of $S(0), S(1)$ and $S(2)$ as the price of an asset today, tomorrow and the day after tomorrow. $X$ and $Y$ might be thought of as the daily logarithmic return. We want to have a predictor of the price of the asset $2$ days from now, given that we observe the price tomorrow. Then $E[S(2)|S(1)]$ gives us the standard estimator, whereas $E_d[S(2)|S(1)]$ gives us the estimator in logarithmic distance.

\section{Sample estimation in the Riemannian metric derived from a Bregman divergence}\label{estimation2}  
In this section we address the issue of sample estimation of the expected values in the $d_\phi$ metric. That is, how to estimate \
\begin{equation}\label{dmean}
E_d[\bX] = H\Big(E[h(\bX)]\big)
\end{equation}
\noindent when all that we have is a sample $\{\bx_1,...,\bx_n\}$ of $\bX.$ The sample estimator is defined to be the point $\bS_n(\{\bx_k\}) \in \cM$ that minimizes the aggregate distance (``cost'' function) 
$$\sum_{k=1}^n d(\bx_n,\bv)^2 = \sum_{k=1}^n \Big(h(\bx_k) - h(\bv)\Big)^2$$
\noindent when $\bv$ ranges over $\cM.$ Clearly, for the geodesic distance computed in (\ref{geodes2}) the minimizer is easy to compute. Again, as $h$ and $H$ are bijections, we have
\begin{equation}\label{est2}
\bS_n(\{\bx_k\}) = H\Big(\frac{1}{n}\sum_{k=1}^n h(\bx_k)\Big).
\end{equation}
Recall that this identity is to be understood componentwise, that is both sides are vectors in $\cM.$\\
Certainly, $d_\phi$-mean (of the set $\{\bx_1,...,\bx_n\}$ is a good name for $\bS_n.$ Given the special form (\ref{dmean}) for $E_d[\bX],$ it is clear that (\ref{est2}) defines an unbiased estimator of the $d_\phi$-mean. At this point we mention that we leave as an exercise for the reader, to  use the semi-parallelogram law to verify the uniqueness of the minimizer of the distance to a set of points given by (\ref{est2}).

But the worth of (\ref{est2}) is for the proof of the law of large numbers. But first, we need to note that the error in estimating $\bX$ by its expected value, that is, the variance of $\bX$ is
\begin{equation}\label{dvar}
\sigma_d^2(\bX) = E[\big(h(\bX) - h(E_d[\bX])\big)^2]
\end{equation}
In this case, as with the standard proof of the weak law of large numbers we have
\begin{theorem}\label{lln}
Suppose that $\{\bX_k : k\geq 1\}$ is an i.i.d $\cM$-valued random variables, that have finite $\sigma^2_d$ variance. Then $\bS_n(\{\bx_k\})\to E_d[\bX]$ in probability.
\end{theorem} 
Proceeding as in the case of Euclidean geometry, we have

\begin{theorem}\label{estvar}
With the same notations and assumptions as in the previous result, define the estimator of the variance by
$$\hat{\sigma}_d^2 = \frac{1}{(n-1)}\sum_{k=1}^n \big(h(\bX_k) - h(\bS_n(\{\bx_k\}))\big)^2.$$
Then $\hat{\sigma}_d^2$ is an unbiased estimator of $\sigma_d^2(\bX).$
\end{theorem}
Comment: Observe that $\hat{\sigma}_d^2$ is a positive, real random variable. So, its expected value is the standard expected value.

\section{Arithmetic properties of the expectation operation}\label{group}
When there is a commutative group operation on $\cM$ that leaves the metric invariant, then the best predictors have additional properties. The two standard examples that we have in mind are $\cM=\mathbb{R}^d$ and the group operation being the standard addition of vectors, or $\cM=(0,\infty)^d$ and the group operation being the component wise multiplication. For definiteness, let us denote the group operation by $\bx_1\circ\bx_2$ and the inverse of $\bx$ with respect to that operation by $\bx^{-1}.$ Let $\be$ denote the identity for that operation. That the distance invariant (or the group operation is an isometry), that is
\begin{equation}\label{ginv}
\mbox{Suppose that for any}\;\;\bx,\,\by,\,\bv \in \cM\;\;\mbox{we have}\;\;\;d(\bx\circ\bv,\by\circ\bv)=d(\bx,\by).
\end{equation}
Some simple consequences of this fact are the following. To begin with, we can define a norm derived from the distance by $|\bx|_d=d(\be,\bx).$
We leave it up to the reader to verify that in this notation the triangle inequality for $d$ becomes $|\bx\circ\by^{-1}|_d=d(\bx,\by) \leq |\bx|_d + |\by|_d,$ and that this implies that implies that $|\bx|_d = |\bx^{-1}|_d.$

Let us now examine two examples of the situation described above. For the first example in Table \ref{tab1}, in which the conditional expectation in divergence and in the distance derived from it coincide, we know that the conditional expectation is linear. In the last example in Table \ref{tab1}, the analogue of multiplication by a scalar is the (componentwise) exponentiation. In this case, we saw that the conditional expectation of a strictly positive random variable $\bX$  with respect to a $\sigma$-algebra $\cG$ is
$$E_d[\bX|\cG] = e^{E[\ln(\bX)|\cG]},$$
It is easy to verify, and it is proved in \cite{GH}, that
\begin{theorem}\label{lin}
Let $\bX_1$ and $\bX_2$ be two $(0,\infty)^d$-valued which are $\bbP$-integrable in the logarithmic metric. Let $a_1$ and $a_2$ be two real numbers, then
$$E_d[\bX_1^{a_1}\bX_2^{a_2}|\cG] = \Big(E_d[\bX_1|\cG]\Big)^{a_1}\Big(E_d[\bX_2|\cG]\Big)^{a_2}.$$
\end{theorem}

\section{Concluding comments}
\subsection{General comments about prediction}
A predictive procedure involves several aspects: To begin with, we have to specify the nature of the set in which the random variables of interest take values and the class of predictors that we are interested in. Next comes the criterion, cost function or error function used to quantify the ``betterness'' of a predictor, and finally, we need some way to decide on the uniqueness of the best predictor.

We mentioned at the outset that using the notion of divergence function, there exists a notion of best predictor for random variables taking values in convex subsets $\cM$ of some $\mathbb{R}^d,$ which, somewhat surprisingly, coincides with the standard least squares best predictor.  The fact that in the Riemannian metric on $\cM$ derived from a divergence function a notion of best predictor exists, suggests the possibility of extending the notion of best predictor to Tits-Bruhat spaces. These are complete metric spaces, whose metric satisfies the semi-parallelogram law stated in Lemma 2.1. Using the completeness of the space the notion of ``mean'' of a finite set as the point that minimizes the sum of the squares of the distances to the points of the set, or that of best predictor are easy to establish. And using the semi-parallelogram law, the uniqueness of the best predictor can be established. 

The best predictors can be seen to have some of the properties of conditional expectation, except those that depend on the underlying vector space structure of $\cM,$ like Jensen's inequality and the ``linearity'' of the best predictor.

\subsection{Other remarks}
In some cases it is interesting to consider the Legendre-Fenchel duals of the convex function generating the divergence, see \cite{BDGMM}, \cite{BNN} or \cite{N} for example. The Bregman divergences induce a dually flat space, and conversely, we can associate a canonical Bregman divergence to a dually flat space.\footnote{Thanks to Frank Nielsen for the remark} The derived metric in this case is the (algebraic) inverse of the original metric, and it generates the same distance, see \cite{AC} for this. Therefore the same comparison results hold true in this case as well. 

As remarked at the end of Section 2.1, to compare the derived metrics to the standard Euclidean metric, and therefore, to compare the prediction errors (or the $d$-variance to the standard variance of a $\cM$ random variable does not seem to be an easy task. This is a pending issue to be settled.
 
We saw that the set $\cM$ on which the random variables of interest may be equipped with more than one distance.
The results presented above open up the door to the following conceptual (or methodological) question: Which is the correct distance to be used to make predictions about $\cM$-valued random variables? 

Other pending issue corresponds to the general case in which $\Phi(\bx)$ is not of the type (\ref{convf}). In this case, by suitable localization we might reproduce the results of Section 2 locally. The problem is to paste together
the representation of the geodesics and the rest using the local representation.

We saw as well that when there is no algebraic structure upon $\cM,$ some properties of the estimators are related only to the metric properties of the space, while when there is a commutative operation on $\cM,$ the best estimators have further algebraic properties. In reference to the examples in Section 2, an interesting question is which metrics admit a commutative group operation that leaves them invariant.

\section{Appendix: Integration of the geodesic equations}
Consider (\ref{geoeq}), that is
$$\phi''(x_i)\ddot{x}_i + \frac{1}{2}\phi'''(x_i)\dot{x}_i^2 = 0$$
This is the Euler-Lagrange equation of the Lagrangian $L(x,\dot{x})=g\dot{x}^2/2,$ where we put $\phi''(x)=g(x).$ 

Notice now, that of we make the change of variables $y=h(x),$ where $h'(x)=g^{1/2}(x),$ in the new coordinates we can write the Lagrangian function as $L(y,\dot{y})=\dot{y}^2.$ In these new coordinates the geodesics are straight lines 
$$y(t) = y(0) + kt.$$
If at $t=0$ the geodesic starts at $x_0$ (or at $y_0=h(x_0)$) and at $t=1$ it is at $x_1$ (or at $y_1=h(x_1)$), we obtain $k=h(x_1)-h(x_0).$ 

{\bf Acknowledgment} I want to thank Frank Nielsen for his comments and suggestions on the first draft of this note.

\end{document}